\documentclass[a4paper,12pt]{amsart}
\usepackage[margin=1.4in]{geometry} 
\usepackage{amsfonts,amssymb,amsmath,amsthm}
\usepackage{enumerate}
\usepackage[shortlabels]{enumitem}
\usepackage{mathabx}             
\usepackage{xargs}               
\usepackage{xfrac} 
\usepackage{bm}    
\usepackage{comment}
\usepackage{xifthen}
\usepackage{graphicx}
\usepackage{float}
\usepackage{wrapfig}
\usepackage{caption}
\usepackage{epsfig}
\usepackage{chngcntr}
\usepackage{nicefrac}
\usepackage{thmtools}
\usepackage{hyperref}
\usepackage[capitalize,noabbrev]{cleveref}

\DeclareMathOperator{\PSL}{PSL}
\DeclareMathOperator{\SO}{SO}
\DeclareMathOperator{\SU}{SU}

\newcommand{\R}{\mathbb{R}}
\newcommand{\C}{\mathbb{C}}

\newcommand{\GmodGamma}{G / \Gamma}
\newcommand{\Hspace}{\mathbb{H}^3}
\newcommand{\FHspace}{\mathcal{F}\Hspace}
\newcommand{\tmu}{\tilde{\mu}}
\newcommand{\M}{\mathcal{M}}

\newcommand{\injM}{\mathrm{inj}_{\M}}

\theoremstyle{plain}
\newtheorem{thm}{Theorem}[section]
\newtheorem{lemma}[thm]{Lemma}
\newtheorem{cor}[thm]{Corollary}

\theoremstyle{definition}
\newtheorem{dfn}{Definition}[section]

\theoremstyle{remark}
\newtheorem{rem}[thm]{Remark}

\numberwithin{equation}{section} 

\newtheoremstyle{TheoremNum}
{\topsep}{\topsep}
{\itshape}        
{}                
{\bfseries}       
{.}               
{ }               
{\thmname{#1}\thmnote{ \bfseries #3}}
\theoremstyle{TheoremNum}
\newtheorem{thmnm}{Theorem}

\begin{document}
	
\graphicspath{ {} }

\pagestyle{plain}
\title{Horospherically invariant measures and finitely generated Kleinian groups}
\author{Or Landesberg}
\thanks{This work was supported by ERC 2020 grant HomDyn (grant no.~833423).}

\begin{abstract}
	Let $ \Gamma\! <\! \PSL_2(\C) $ be a Zariski dense finitely generated Kleinian group. We show all Radon measures on $ \PSL_2(\C) / \Gamma $  which are ergodic and invariant under the action of the horospherical subgroup are either supported on a single closed horospherical orbit or quasi-invariant with respect to the geodesic frame flow and its centralizer. We do this by applying a result of Landesberg and Lindenstrauss \cite{LL2019Radon} together with fundamental results in the theory of 3-manifolds, most notably the Tameness Theorem by Agol~\cite{Agol} and Calegari-Gabai~\cite{Calegari-Gabai}.
\end{abstract}

\maketitle

\section{Introduction}

Let $ G=\PSL_2(\C) $ be the group of orientation preserving isometries of $ \Hspace $, equipped with a right-invariant metric. Let $ \Gamma < G $ be a discrete subgroup (a.k.a.~a Kleinian group) and let
\begin{equation*}
U = \left\{u_z = \begin{pmatrix} 1 & z \\ & 1 \end{pmatrix} \; :\; z \in \C \right\}
\end{equation*}
be a horospherical subgroup. We are interested in understanding the $ U $-ergodic and invariant Radon measures (e.i.r.m.) on $ \GmodGamma $. 
\medskip

Unipotent group actions, and horospherical group actions in particular, exhibit remarkable rigidity. There are only very few finite $ U $-ergodic and invariant measures as implied by Ratner's Measure Rigidity Theorem~\cite{Ratner1991} --- indeed if $ \Gamma < G $ is a discrete subgroup the only finite measures on $\GmodGamma$ that are $U$-ergodic and invariant are those supported on a compact $ U $-orbit and possibly also Haar measure if $ \GmodGamma$ has finite volume. This result was predated in special cases by Furstenberg \cite{Furstenberg1973}, Veech \cite{Veech1977} and Dani \cite{Dani1978}.

This scarcity of ergodic invariant measures with respect to the action of the horospherical subgroup extends to infinite invariant Radon measures under the additional assumption that $ \Gamma $ is geometrically finite. In this setting the only $ U $-e.i.r.m.~not supported on a single $ U $-orbit is the Burger-Roblin measure, as shown by Burger \cite{Burger1990}, Roblin \cite{Roblin2003} and Winter \cite{Winter2015}.

When considering infinite invariant Radon measures over geometrically infinite quotients this uniqueness phenomenon breaks down. It was discovered by Babillot and Ledrappier \cite{Babillot-Ledrappier} that abelian covers of compact hyperbolic surfaces support an uncountable family of horocycle invariant ergodic and recurrent Radon measures. A characteristic feature of this family of invariant measures is quasi-invariance with respect to the geodesic frame flow, a property also exhibited by the Burger-Roblin measure in the geometrically finite context. We note that such quasi-invariant $ U $-e.i.r.m.~scorrespond to  $ \Gamma $-conformal measures on the boundary of hyperbolic space, see \S\ref{subsec_Gamma conf meas}.

A natural question arises --- do all non-trivial $ U $-e.i.r.m.~on $ \GmodGamma $ correspond to $ \Gamma $-conformal measures in this manner? An affirmative answer was given in several different setups by Sarig \cite{Sarig2004,Sarig2010}, Ledrappier \cite{Ledrappier2008, Ledrappier-Sarig} and Oh and Pan \cite{Oh-Pan}. Particularly, in the context of geometrically infinite hyperbolic 3-manifolds, Ledrappier's result \cite{Ledrappier2008} implies any locally finite measure which is ergodic and invariant w.r.t.~the horospherical foliation on the unit tangent bundle of a regular cover of a compact 3-manifold is also quasi-invariant w.r.t.~the geodesic flow. This was later extended in the case of abelian covers to the full frame bundle by Oh and Pan \cite{Oh-Pan}.

Inspired by Sarig \cite{Sarig2010}, the measure rigidity results above were significantly extended in \cite{LL2019Radon} to a wide variety of hyperbolic manifolds, including as a special case all regular covers of geometrically finite hyperbolic $ d $-manifolds.
\medskip

In this paper we apply a geometric criterion in \cite{LL2019Radon} for quasi-invariance of a $ U $-e.i.r.m.~to prove measure rigidity in the setting of geometrically infinite hyperbolic $ 3 $-manifolds with finitely generated fundamental group. This is only possible due to the deep understanding of the geometry of such manifolds and their ends, specifically the fundamental results of Canary \cite{Canary1996}, and the proof of the Tameness Conjecture by Agol \cite{Agol} and Calegari-Gabai \cite{Calegari-Gabai}. 

\subsection{Statement of the main theorem}
Following Sarig, we call a $ U $-ergodic and invariant Radon measure on $ \GmodGamma $ \emph{trivial} if it is supported on a single proper horospherical orbit of one of the following two types:
\begin{enumerate}
	\item a horospherical orbit based outside the limit set, and thus wandering to infinity through a geometrically finite end (see \S\ref{Sec_Ends of 3-manifolds});
	\item a horospherical orbit based at a parabolic fixed point, and thus bounding a cusp.
\end{enumerate}

A measure $ \mu $ is called \emph{quasi-invariant} with respect to an element $ \ell \in G $ if $ \ell.\mu \sim \mu $, i.e.~if the action of $ \ell $ on $ \GmodGamma $ preserves the measure class of $ \mu $. We say $ \mu $ is $ H $-quasi-invariant with respect to a subgroup $ H < G $ if it satisfies the above criterion for all elements $ \ell \in H $. Denote by $ N_G(U) $ the normalizer of $ U $ in $ G $. 

We show the following:
\begin{thm}\label{Main_thm}
	Let $ \Gamma < G $ be a Zariski dense finitely generated Kleinian group. Then any non-trivial $ U $-e.i.r.m.~on $ \GmodGamma $ is $ N_G(U) $-quasi-invariant.
\end{thm}

Note that this theorem implies in particular that the only proper horospherical orbits in $ \GmodGamma $ are trivial (otherwise such orbits would support a non-trivial $ U $-e.i.r.m.~given as the pushforward of the Haar measure of $ U $). Cf.~\cite{Bellis, Lecuire_Mj, Ledrappier1998, Matsumoto} for related results on the topological properties of horospherical orbits in the geometrically infinite setting.
\medskip

The subgroup $ U $ is the unstable horospherical subgroup w.r.t.
\[ A=\left\{a_t = \begin{pmatrix} e^{\nicefrac{t}{2}} & \\ &  e^{\nicefrac{-t}{2}} \end{pmatrix} \; :\; t \in \R \right\}, \]
the $ \R $-diagonal Cartan subgroup. We denote by
\[ K = \SU(2) / \{ \pm I \} \cong \SO(3) \]
the maximal compact subgroup fixing the origin in $ \Hspace $. Set
\[  M=Z_K(A) = \left\{\begin{pmatrix} e^{i\theta} & \\ &  e^{-i\theta} \end{pmatrix} \; :\; \theta \in \R \right\} \]
the centralizer of $ A $ in $ K $. Note that $ N_G(U) = MAU $, hence $ N_G(U) $-quasi-invariance in the statement of \cref{Main_thm} may be replaced with $ MA $-quasi-invariance.
\medskip

The subgroup $ P = MAU = N_G(U) $ is a minimal parabolic subgroup of $ G $. We may draw the following simple corollary of our main theorem:
\begin{cor}\label{Cor_P measures}
	Let $ \Gamma < G $ be a Zariski dense finitely generated Kleinian group. Then any $ U $-invariant Radon measure on $ \GmodGamma $ which is $ P $-quasi-invariant and ergodic, is also $ U $-ergodic unless it is supported on a single $P$-orbit fixing a parabolic fixed point or a point outside the limit set of $\Gamma$.
\end{cor}

\begin{rem}
	One implication of this corollary is the $ U $-ergodicity of Haar measure on $ \GmodGamma $ whenever $ \Gamma $ is finitely generated and of the first kind (having a limit set equal to $ \partial \Hspace $). This follows from the fact that Haar measure is $ P $-invariant and under these assumptions also $ MA $-ergodic (see \cite[Theorem 9.9.3]{Thurston}). As pointed out to us by Hee Oh, this can also be derived using different means (applicable in more general settings), see \cite[Theorem 7.14]{Lee-Oh}.
\end{rem}

\subsection{$ MA $-quasi-invariance} 
Given a discrete subgroup $ \Gamma < G $, the space $ K \backslash \GmodGamma $ is naturally identified with a corresponding 3-dimensional orbifold  $ \Hspace / \Gamma $. The space $ \GmodGamma $ is identified with its frame bundle\footnote{Strictly speaking, this is the case when $ \Hspace / \Gamma $ is a manifold, and can be used to define the frame bundle when $ \Hspace / \Gamma $ is an orbifold.} $ \FHspace / \Gamma $. The left action of $ A $ on $ \GmodGamma $ corresponds to the geodesic frame flow on $ \FHspace / \Gamma $ and the action of $ M $ corresponds to a rotation of the frames around the direction of the geodesic.
\cref{Main_thm} relies on a geometric criterion for $ MA $-quasi-invariance developed in \cite{LL2019Radon} and described below.
\medskip

Given $ x=g\Gamma \in \GmodGamma $ denote
\[ \Xi_{x} = \bigcap_{n=1}^\infty \overline{\bigcup_{t \geq n} a_{-t}g\Gamma g^{-1}a_t}, \]
the set of accumulation points in $ G $ of all sequences $ x_n \in a_{-t_n}g\Gamma g^{-1}a_{t_n} $ for $ t_n \to \infty $, sometimes also denoted as $ \limsup_{t \to \infty} \left(a_{-t}g\Gamma g^{-1}a_t\right) $.

Recall that $ a_{-t}g\Gamma g^{-1}a_t $ is the stabilizer of the point $ a_{-t}x $ in $ \GmodGamma $. Hence $ \Xi_{x} $ may be viewed as the set of accumulation points of elements of $ \Gamma $ as ``seen'' from the viewpoint of the geodesic ray $ \left\{a_{-t}x\right\}_{t\geq 0} $. Using Thurston's notion of a ``geometric limit'', we may say the set $ \Xi_x $ contains all geometric limits of subsequences of the family $ (a_{-t}g\Gamma g^{-1}a_t)_{t \geq 0} $.

We shall make use of the following sufficient condition for $ MA $ quasi-invariance established in arbitrary dimension in \cite{LL2019Radon}:

\begin{thm}\label{Thm_MA-qi}
	Let $ \Gamma < G $ be any discrete subgroup and let $ \mu $ be any $ U $-e.i.r.m. on $ \GmodGamma $. If $ \Xi_x $ contains a Zariski dense subgroup of $ G $ for $ \mu $-a.e.~$ x \in \GmodGamma $, then $ \mu $ is $ MA $-quasi-invariant.
\end{thm}

Understanding the ``ends'' of a 3-manifold allows to draw conclusions regarding the different $ \Xi_{x} $ observed along divergent geodesic trajectories. The strategy of proof in \cref{Main_thm} is to import the necessary knowledge regarding the geometry of 3-manifolds with finitely generated fundamental group, in order to show the conditions of \cref{Thm_MA-qi} are satisfied for all non-trivial measures.

In principle the technique used in the proof of \cref{Main_thm} is applicable to any dimension. Let $ G_d = \SO^+(d,1) $ denote the group of orientation preserving isometries of hyperbolic $ d $-space, for any $ d \geq 2 $. Set:
\begin{itemize}[leftmargin=*]
	\item $ A=\{a_t\}_{t \in \R} $ --- Cartan subgroup of $ \R $-diagonal elements.
	\item $ U = \{ u : a_{-t}ua_{t} \to e \text{ as } t \to \infty \} $ --- the unstable horospherical subgroup.
	\item $ K \cong SO(d) $ --- stabilizer of the origin in $ \mathbb{H}^d $ (maximal compact subgroup).
	\item $ M=Z_K(A) \cong \SO(d-1) $ --- the centralizer of $ A $ in $ K $.
\end{itemize} 

Then an adaptation of the proof of \cref{Main_thm} can be used to show the following:
\begin{thm}\label{Thm_d-dim}
	Let $ \Gamma < G_d $ be a discrete group with a uniform upper bound on the injectivity radius at all points in $ G_d / \Gamma $. Then any $ U $-e.i.r.m.~$ \mu $ on $ G_d / \Gamma $ is either supported on a single proper horospherical orbit bounding a cusp or is $ MA $-quasi-invariant. 
\end{thm}
For $ d = 3 $, the geometric condition above is not satisfied by all finitely generated Kleinian groups, rather only by those of the first kind. The proof of \cref{Main_thm} takes advantage of further geometric properties known to be satisfied by all 3-manifolds with finitely generated fundamental group.

\subsection{Measure decomposition and $ \Gamma $-conformal measures}\label{subsec_Gamma conf meas}

Before proceeding to the main part of the paper, we briefly present the implications of $ MA $-quasi-invariance in terms of the structure of $ U $-e.i.r.m. and the relation to $ \Gamma $-conformal measures on the boundary of $ \Hspace $.
\medskip

A probability measure $ \nu $ on $ S^2 \cong \partial \Hspace $ is called $ \Gamma $-conformal of dimension $ \beta > 0 $ if for any $ \gamma \in \Gamma $
\[ \frac{d(\gamma_* \nu)}{d\nu} = \|\gamma'\|^\beta \]
where $ \gamma_* \nu $ is the push-forward of $ \nu $ by the action of $ \gamma $ on $ S^2 $ and $ \|\cdot\| $ is the operator norm on $ TS^2 $ with respect to the standard Riemannian metric.
Let $ \mu $ be a $ U $-e.i.r.m.~on $ \GmodGamma $ and let $ \tmu $ be its lift to $ G $. A well known consequence of $ MA $-quasi-invariance is a decomposition of $ \tmu $ in $ S^2 \times M \times A \times U  $ coordinates given by 
\begin{equation}\label{eq_decomposion}
d\tmu = e^{\beta t} d\nu dm dt du,
\end{equation}
where $ \nu $ is a $ \Gamma $-conformal measure of dimension $ \beta $ on $ S^2 $ and $ dm,\, dt,\, du $ denote Haar measures on $ M,\, A $ and $ U $ respectively, see \cite{Babillot2004,Burger1990,Ledrappier-Sarig,Sarig2010} and also \cite[\S 5.2]{LL2019Radon}. If $ \delta(\Gamma) $ denotes the critical exponent of the discrete group $ \Gamma $, then all $ \Gamma $-conformal measures are of dimension $ \beta \geq \delta(\Gamma) $ unless $ \Gamma $ is a parabolic cyclic group, see \cite[Theorem 2.19]{Sullivan1987}.

The $ U $-ergodicity of $ \mu $ together with $ MA $-quasi-invariance implies that there exists an $ \alpha \in \R $ satisfying
\[ ma_t.\mu = e^{\alpha t} \mu \]
for any $ ma_t \in MA $. This parameter is related to the dimension of conformality via the identity $ \alpha=\beta-2 $. Note that $ \mu $ is in fact $ M $-invariant since $ M $ is compact and has no non-trivial real-valued characters.

Arguing in the other direction, any $ \Gamma $-conformal measure on $ S^2 $ induces a $ U $-invariant $ MA $-quasi-invariant Radon measure on $ \GmodGamma $, via \eqref{eq_decomposion}.
Consequently, \cref{Main_thm} gives a one-to-one correspondence
\begin{equation*}
\left\{ \text{non-trivial $ U $-e.i.r.m.s} \right\} \rightleftarrows \left\{  \begin{matrix}
\text{non-atomic ergodic} \\
\Gamma \text{-conformal measures}
\end{matrix} \right\}
\end{equation*} 
where ergodicity of the $ \Gamma $-conformal measures is with respect to the entire $ \Gamma $-action on $ S^2 $.

In several setups such classification has been be further explicitized. For instance when $ \Gamma $ is geometrically finite, Sullivan proved there is only one atom-free $ \Gamma $-conformal measure supported on the limit set \cite{Sullivan1984}, hence the induced measure on $ \GmodGamma $ (the Burger-Roblin measure) is the only non-trivial $ U $-e.i.r.m. and is in particular ergodic (this result was originally proven by Burger \cite{Burger1990}, Roblin \cite{Roblin2003} and Winter \cite{Winter2015}, for yet another proof see \cite[Thm.~5.7]{Mohammadi-Oh}).

In the geometrically infinite setting, we refer the reader to relevant results by Bishop and Jones \cite{Bishop-Jones} and Anderson, Falk and Tukia \cite{Anderson-Falk-Tukia} who give explicit constructions and partial classification results applicable also to the context of this paper.

\subsection{Structure of the paper} In the section to follow, \S \ref{Sec_Ends of 3-manifolds}, we present a self-contained introduction to the theory of hyperbolic 3-manifolds with finitely-generated fundamental group and reference those results instrumental to the proof of \cref{Main_thm}. Proofs of \cref{Main_thm} and \cref{Cor_P measures} are detailed in the final section of the paper.

\subsection{Acknowledgments}
The author would like to thank Yair Minsky for several invaluable conversations introducing the theory of 3-manifolds as well as some important insights used in the proof of \cref{Main_thm}. 
The author would also like to thank Peter Sarnak for introducing him to the Tameness Theorem and suggesting its relevance, and to Hee Oh and the anonymous referees for their helpful comments reviewing earlier versions of this manuscript. 

This paper is part of the author's PhD thesis conducted at the Hebrew University of Jerusalem under the guidance of Prof. Elon Lindenstrauss. The author is in debt to Elon for his continuous support and for many helpful suggestions.

\section{Ends of finitely generated Kleinian groups}\label{Sec_Ends of 3-manifolds}

Let $ \Gamma < G $ be a torsion-free Kleinian group and let $ \M = K \backslash \GmodGamma = \Hspace / \Gamma $ be the hyperbolic 3-manifold defined by $ \Gamma $. In this section we will describe a classical decomposition of $ \M $ into a compact set and a collection of unbounded components, or \emph{ends}, capturing the different ways in which a geodesic trajectory can ``escape to infinity''. Following several fundamental results in the theory of 3-manifolds, we will give a very rough but nonetheless useful characterization of these ends which would serve us in the proof of \cref{Main_thm}. Everything in this section is well known and may be found in several sources, see \cite{Canary1993,Canary1996,Marden,Matsuzaki-Taniguchi,Minsky,Thurston,Thurston_book_vol1}.

\subsection{Convex core}
The \emph{limit set} $ \Lambda \subseteq \partial \Hspace = S^2 $ of $ \Gamma $ is defined as the set of accumulation points of the orbit $ \Gamma.z $ in $ \overline{\Hspace} $ for some (and any) point $ z \in \Hspace $. The complement $ \Omega = \partial \Hspace \smallsetminus \Lambda $ is called the \emph{domain of discontinuity} and may equivalently be defined as the maximal set in $ \partial \Hspace $ for which the action of $ \Gamma $ is properly discontinuous.

Let $ \mathrm{hull}(\Lambda) $ denote the convex hull of $ \Lambda $ inside $ \Hspace $. Define the \emph{convex core} of $ \M $ as $ C(\M) = \mathrm{hull}(\Lambda) / \Gamma $. Alternatively, this is the minimal convex submanifold of $ \M $ for which the inclusion induces an isomorphism 
\[ \pi_1(C(\M)) \hookrightarrow \pi_1(\M). \]

\subsection{The thick/thin decomposition and the non-cuspidal part of $ \M $}\label{subsec_noncuspidal}
Assume $ \Gamma $ contains parabolic elements and let $ \xi \in \partial \Hspace $ be a parabolic fixed point. Denote by $ \Gamma_\xi $ the stabilizer subgroup of $ \xi $ in $ \Gamma $. The group $ \Gamma_\xi $ is a free abelian group of rank one or two\footnote{In general, when $ \Gamma $ is not assumed to be torsion-free, this statement only holds virtually.}. There exists an $ r > 0 $ for which any open horoball $ B $ of Euclidean radius less than $ r $ based at $ \xi $ satisfies that $ B / \Gamma_\xi $ is embedded in $ \M $. Such an open set $ B / \Gamma $ in $ \M $ is called a \emph{cusp neighborhood} (of $ \xi $). The rank of the cusp is the rank of $ \Gamma_\xi $.
\medskip

The \emph{injectivity radius}, $ \injM(x) $, at a point $ x \in \M $ is the supremal radius of an embedded hyperbolic ball in $ \M $ around $ x $. Equivalently, it is equal half the length of the shortest homotopically non-trivial loop through $ x $, or in the language of $ K \backslash \GmodGamma $
\begin{equation}\label{eq_injGamma}
\injM(Kg\Gamma) = \frac{1}{2}\min \left\{ d_{\Hspace} (Ke,Kh) : h \in g\Gamma g^{-1}\smallsetminus\{e\} \right\}
\end{equation}
where we identify $ \Hspace \cong K \backslash G $.

Given $ \epsilon > 0 $, we may decompose the manifold $ \M $ into its thick and thin parts, i.e.
\[ \M_{thick(\epsilon)} = \{ x \in \M : \injM (x) \geq \epsilon \} \] and its complement $ \M_{thin(\epsilon)} $.

There exists a constant $ \rho_3 > 0 $ called the Margulis constant corresponding to $ \Hspace $ (or to $ G $) for which given any $ 0 < \epsilon < \rho_3 $ all connected components of $ \M_{thin(\epsilon)} $ are either:
\begin{enumerate}
	\item neighborhoods of a closed geodesic (solid tori homeomorphic to $ D\times S^1 $ where $ D $ is the closed unit disc in $ \R^2 $);
	\item neighborhoods of a rank two cusp (known as solid cusp tori); or
	\item neighborhoods of a rank one cusp (respectively, solid cusp cylinders).
\end{enumerate}
A similar classification of the thin components holds for general $ d $-dimensional hyperbolic manifolds, see e.g.~\cite[Theorem 4.5.6]{Thurston_book_vol1}.

Fix $ 0 < \epsilon < \rho_3 $ and let $ \mathcal{P} \subset \M $ denote the union of all components of $ \M_{thin(\epsilon)} $ corresponding to cusps. Let $ \M^0 = \M \smallsetminus \mathcal{P} $ be the \emph{non-cuspidal part} of the manifold $ \M $.

A manifold $ \M $ is called \emph{geometrically finite} if $ C(\M) \cap \M^0 $ is compact (see equivalent definitions in \cite[Theorem 12.4.5]{Ratcliffe}). In this setting of 3-dimensional hyperbolic manifolds, the above definition is also equivalent to $ \Gamma $ having a finite sided fundamental polyhedron in $ \Hspace $ (see \cite[Theorem 12.4.6]{Ratcliffe}). 
\medskip

The following simple lemma will be of use:
\begin{lemma}\label{lemma_unbounded inj rad}
	Let $ \Gamma < G $ be a Kleinian group of the second kind, i.e.~having limit set $ \Lambda \neq S^2 $, and let $ \M = K \backslash G / \Gamma $. Then
	\[ \sup_{x \in \M} \injM (x) = \infty. \]
\end{lemma}

\begin{proof}
	Denote $ \Omega = S^2 \smallsetminus \Lambda \neq \emptyset $ and choose some point $ \xi \in \Omega $ which is not a fixed point of an elliptic element of $ \Gamma $. By the proper discontinuity of the action of $ \Gamma $ on $ \Omega $ there exists a small disc $ D \subset \Omega $ containing $ \xi $ satisfying
	\[ D \cap \gamma.D = \emptyset \quad \text{for all }\gamma \in \Gamma, \]
	see e.g. \cite[Theorems 12.2.8-9]{Ratcliffe}.
	This implies that the half-space $ H \subset \Hspace $ bounded by $ D $ (in other words, the convex hull of $D$ in $\Hspace$) also has a pairwise disjoint $ \Gamma $-orbit, implying $ H $ is embedded in $ \M $. Since the injectivity radius is clearly not bounded from above on $ H $ this implies the claim.
\end{proof}

\begin{rem}
	This claim is stated for dimension $ d=3 $ but can be identically argued for any discrete $ \Gamma < G_d $ with limit set $ \Lambda \neq S^{d-1} $ and $ d \geq 2 $. 
\end{rem}

\subsection{Relative compact core and relative ends}\label{subsec_RelativeEnds}
A \emph{compact core} of a manifold $ \M $ is an irreducible connected compact submanifold $ C_{cpt} \subseteq \M $ for which the inclusion induces an isomorphism $ \pi_1 (C_{cpt}) \hookrightarrow \pi_1 (\M) $. Given a compact core $ C_{cpt} $ of $ \M $, the \emph{ends} of $ \M $ are the connected components of $ \M \smallsetminus C_{cpt} $. 
For a geometrically finite manifold, the non-cuspidal part of the convex core, i.e.~$ C(\M) \cap \M^0 $, is a compact core. 

In general one has the following theorem:

\begin{thm}[McCullough \cite{McCullough}, Scott \cite{Scott}]
	Let $ \M $ be a hyperbolic 3-manifold with finitely generated fundamental group and let $ \mathcal{P} $ be its cuspidal part, as above. There exists a connected compact core $ C_{rel} \subseteq \M $ satisfying:
	\begin{enumerate}
		\item Each torus component of $ \partial \mathcal{P} $ is a component of $ \partial C_{rel} $,
		\item Each cylinder component of $ \partial \mathcal{P} $ intersects $ \partial C_{rel} $ in a closed annular region (homeomorphic to $ S^1 \times [0,1] $).
	\end{enumerate} 
\end{thm}

The submanifold $ C_{rel} $ is called the \emph{relative compact core} of $ \M $. The connected components of $ \M^0 \smallsetminus C_{rel} $ are called the \emph{relative ends} of $ \M $. 
\medskip

\begin{figure}[h]
	\includegraphics[scale=0.4]{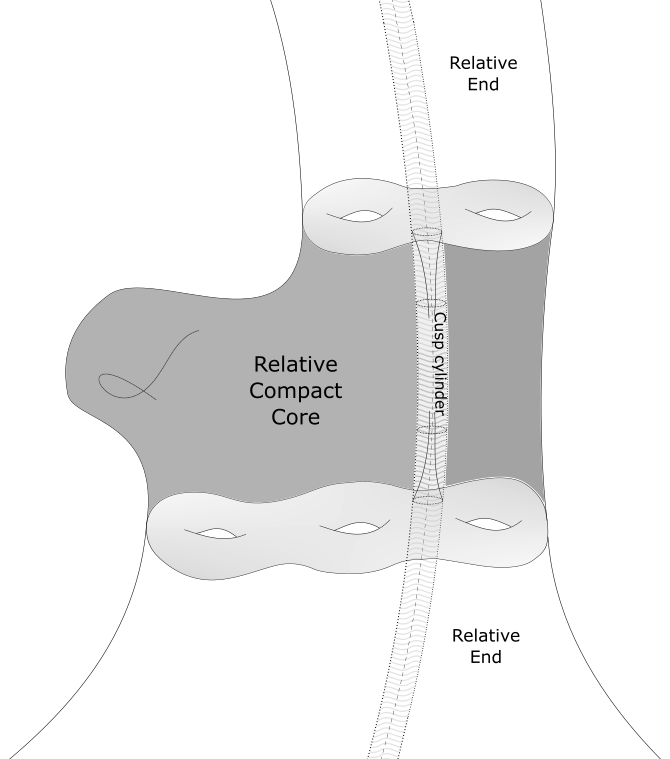}
	\caption{Two relative ends paired by a cusp cylinder.}
	\label{fig:Ends}
\end{figure}
\medskip

\begin{dfn}
	\begin{enumerate}[leftmargin=*]
		\item An open set $ V $ is called a sub-neighborhood of a relative end $ E $ if $ V \subseteq E $ and $ E \smallsetminus V $ is precompact in $ \M $.
		\item A relative end $ E $ of $ \M $ is called \emph{geometrically finite} if it has a sub-neighborhood disjoint from the convex core. A relative end is called \emph{geometrically infinite} otherwise.
\end{enumerate}\end{dfn}
\noindent Note that $ \M $ is geometrically finite if and only if so are all of its relative ends.

\begin{rem}\
	\begin{enumerate}[leftmargin=*]
		\item One considers this refined notion of a relative end, as opposed to taking the connected components of $ \M \smallsetminus C_{cpt} $ (for any compact core $ C_{cpt} $), due to the fact that two relative ends might be connected together by a common cusp cylinder despite having sub-neighborhoods arbitrarily distant from each other (see \cref{fig:Ends}).
		\item The cuspidal neighborhoods in $ \mathcal{P} $ and the relative compact core are not uniquely defined (and are clearly dependent on $ \epsilon $). We will refer to these notions in definite form only to indicate our choices are fixed.
	\end{enumerate}
\end{rem}
\medskip

The following is well known (see e.g.~\cite{Minsky}):
\begin{lemma}\label{lemma_infinite end inside convex core}
	Let $ \M $ be a hyperbolic 3-manifold with finitely generated fundamental group and let $ E $ be a geometrically infinite relative end of $ \M $. Then $ E $ has a sub-neighborhood completely contained inside the convex core.
\end{lemma}

\begin{proof}
	It is a consequence of the Ahlfors finiteness theorem that the boundary of the convex core, $ \partial C(\M) $, has finite hyperbolic area (see \cite[Theorems 4.93 and 4.108]{Kapovich}). Therefore the only non-compact components of $ \partial C(\M) $ are those tending to infinity through a cusp. In particular $ \partial C(\M) \cap \M^0 $ is compact. Hence given a relative end $ E \subseteq \M^0 $ we are ensured that exactly one component of $ E \smallsetminus \partial C(\M) $ is unbounded. In the case that $ E $ is geometrically infinite, this unbounded component is a sub-neighborhood of $ E $ entirely contained inside $ C(\M) $.
\end{proof}

Note that whenever $ \M $ has a relative compact core, $ \M $ admits a finite number of relative ends, as $ \partial C_{rel} $ is compact.

\subsection{Tameness}
The manifold $ \M $ is called \emph{topologically tame} if it is homeomorphic to the interior of a compact 3-manifold with boundary. The following is a fundamental result in the theory of 3-manifolds:
\begin{thm}[The Tameness Theorem]
	All complete hyperbolic 3-manifolds with finitely generated fundamental group are topologically tame.
\end{thm}

The above has been a long standing conjecture of Marden, finally solved in 2004 independently by Agol \cite{Agol} and Calegari-Gabai \cite{Calegari-Gabai}.

The theory of finitely generated Kleinian groups is rich and powerful. In particular much is known regarding the structure and classification of the ends of tame manifolds. We will only make use of a small piece of this theory, encapsulated in the following theorem of Canary \cite[Corollary A]{Canary1996}. 
\begin{thm}\label{Canary's corollary}
	Let $ \M $ be a topologically tame hyperbolic 3-manifold, then there exists a constant $ \kappa > 0 $ such that $ \injM (x) < \kappa $ for any $ x \in C(\M) $.
\end{thm}

It is worth stressing that $ C(\M) $ is not necessarily compact (only whenever $ \M $ is convex cocompact in which case the statement is trivially true). Theorem~\ref{Canary's corollary} plays a key role in the proof of \cref{Main_thm}.

\section{Proof of \cref{Main_thm}}

\begin{thmnm}[\ref{Main_thm}]
	Let $ \Gamma < G $ be a Zariski dense finitely generated Kleinian group. Then any non-trivial $ U $-e.i.r.m.~on $ \GmodGamma $ is $ MA $-quasi-invariant.
\end{thmnm}

\begin{proof}

We claim it suffices to prove the theorem under the additional assumption that $ \Gamma $ is torsion-free. Indeed, by Selberg's lemma (see e.g.~\cite{Matsuzaki-Taniguchi,Ratcliffe}) any finitely generated Kleinian group $ \Gamma $ contains a finite-index torsion-free normal subgroup $ \Gamma' \lhd \Gamma $. Hence there exists a covering map $ \pi $ from $ \GmodGamma' $ to $ \GmodGamma $, with finite fiber. This map is equivariant with respect to the left action of $ G $ on both spaces.

Any $ U $-e.i.r.m.~defined on $ \GmodGamma' $ can be pushed forward to a $ U $-e.i.r.m.~on $ \GmodGamma $, where local finiteness of the projected measure follows from the map $ \pi $ being proper. On the other hand, we claim that any $ U $-e.i.r.m.~$ \mu $ on $ \GmodGamma $ may be lifted (possibly in more than one way) to a $ U $-e.i.r.m.~$ \tilde{\mu} $ on $ \GmodGamma' $ satisfying $ \pi_* \tilde{\mu} = \mu $. Indeed, one can always take $ \nu $ to be the lift of $ \mu $ to a $ \Gamma / \Gamma' $-invariant and $ U $-invariant Radon measure on $ \GmodGamma' $ by defining 
\[ \nu(B) = \sum_{\xi \in \Gamma / \Gamma'} \mu(\pi(B\cap F\xi))  \]
for any measurable set $ B \subseteq \GmodGamma' $, where $ F \subseteq \GmodGamma' $ is a measurable fundamental domain with respect to the action of the group $ \Gamma / \Gamma' $. Whenever $ \nu $ is $ U $-ergodic take $ \tilde{\mu} = \frac{1}{|\Gamma/ \Gamma'|}\nu $. Otherwise, let $ Q \subseteq \GmodGamma' $ be a $ U $-invariant set of positive $ \nu $-measure and consider the finite partition of $ \pi^{-1}(\pi(Q)) $ generated by the $ \Gamma / \Gamma' $-translates of $ Q $, that is, the set of atoms of the algebra generated by translates of $ Q $ in $ \pi^{-1}(\pi(Q)) $. Let $ Q_0 $ be an atom of the partition having positive $ \nu $-measure. Then $ Q_0 $ is $ U $-invariant as an intersection of $ U $-invariant sets (since the actions of $ \Gamma / \Gamma' $ and $ U $ commute). Moreover, all $ \Gamma / \Gamma' $-translates of $ Q_0 $ are mutually disjoint meaning that $ Q_0 $ projects injectively onto $ \pi(Q_0) $. The set $ \pi(Q_0) $ is both $ U $-invariant and of positive $ \mu $-measure, hence by the ergodicity of $ \mu $ it is $ \mu $-conull. The measure 
\[ \tilde{\mu}= \nu |_{Q_0} = \left((\pi|_{Q_0})^{-1}\right)_* \mu  \]
is an ergodic lift of $ \mu $, as claimed.

Equivariance of $ \pi $ ensures both that the push forward of an $ MA $-quasi-invariant measure is $ MA $-quasi-invariant and that a partially periodic horosphere (i.e.~ one which is based at a parabolic fixed point) is projected to a partially periodic horosphere.
Lastly, recall that the limit set of $ \Gamma' $ is equal to the limit set of $ \Gamma $ (true for all normal subgroups, see e.g. \cite[Theorem 12.2.14]{Ratcliffe}) therefore horospheres based outside the limit set are projected by $ \pi $ to those based outside the limit set.  

Hence proving the theorem for $ \GmodGamma' $ would ensure that any $ U $-ergodic lift of a $ U $-e.i.r.m.~$ \mu $ on $ \GmodGamma $, is either a trivial measure or $ MA $-quasi-invariant. Projecting back to $ \GmodGamma $ would imply the same dichotomy for $ \mu $, concluding the reduction claim.
\medskip

Let $ \Gamma < G $ be a finitely generated torsion-free Kleinian group and let $ \mu $ be any $ U $-e.i.r.m.~on $ \GmodGamma $. Let $ \M = K \backslash \GmodGamma = \Hspace / \Gamma $ be the corresponding topologically tame hyperbolic 3-manifold. Let $ p: \GmodGamma \to \M $ denote the projection $ g\Gamma \mapsto Kg\Gamma $.

Recall that for any $ u \in U $ and $ g \in G $
\[ d_G(a_{-t}g,a_{-t}ug) \to 0 \text{ as } t \to \infty, \]
where $ d_G $ denotes a right-invariant metric on $ G $. Consequently, the function
\[ g\Gamma \mapsto \limsup_{t \to \infty} \injM (p(a_{-t}g\Gamma)) \]
is $ U $-invariant. This function is clearly also measurable and by ergodicity it is equal to a constant $ I_\mu \in [0,\infty] $ $ \mu $-a.s. We will consider three possible cases: 
\[ I_\mu = 0 \;,\; I_\mu = \infty \text{ and } 0 < I_\mu <\infty. \]

Whenever $ I_\mu = 0 $ then for $ \mu $-a.e.~$ x=g\Gamma \in \GmodGamma $ the injectivity radius along the geodesic ray $ \left( p(a_{-t}x) \right)_{t \geq 0} $ tends to 0. This can only happen whenever the geodesic is directed toward a cusp. Indeed for all large $ t $ the injectivity radius is less than $ \rho_3 $, the Margulis constant (see \S \ref{subsec_noncuspidal}), ensuring the geodesic is trapped inside one connected component of $ \M_{thin(\rho_3)} $. By the classification of thin components presented in \S \ref{subsec_noncuspidal} either the geodesic is trapped inside a compact toral neighborhood of a short geodesic or it is trapped inside a cusp neighborhood. But since the injectivity radius tends to 0 the former possibility is excluded implying the geodesic tends to infinity through a cusp. In such case the measure $ \mu $ is supported on horospheres based at parabolic fixed points. By ergodicity $ \mu $ has to be trivial, specifically supported on a single horosphere bounding a cusp.

In the case where $ I_\mu = \infty $ we can deduce from Canary's theorem (\cref{Canary's corollary}) that for $ \mu $-a.e.~$ x \in \GmodGamma $ there exists a subsequence of times in which $ \left( p(a_{-t}x) \right)_{t \geq 0} $ tend arbitrarily far away from the convex core $ C(\M) $, implying the geodesic rays must end outside the limit set, meaning $ \mu $-a.e.~$ U $-orbit is wandering. Ergodicity once more ensures the measure is trivial, specifically supported on a single wandering horosphere.
\medskip

Let us now assume that $ 0 < I_\mu < \infty $. We will show that in such case the measure $ \mu $ is $ MA $-quasi-invariant by showing the set
\[ \Xi_{g\Gamma} =\bigcap_{n=1}^\infty \overline{\bigcup_{t \geq n} a_{-t}g\Gamma g^{-1}a_t} \]
contains a Zariski-dense subgroup for $ \mu $-a.e.~$ g\Gamma $ and applying \cref{Thm_MA-qi}. 

Given any point $ g\Gamma $, the geodesic trajectory $ \left( p(a_{-t}g\Gamma) \right)_{t \geq 0} $ satisfies one of two possibilities, either
\begin{enumerate}
	\item the geodesic trajectory is recurrent, i.e.~returns infinitely often to some compact subset of $ \M $ (whenever the geodesic is facing a conical limit point\footnote{This is in fact the defining property of conical limit points.}); or
	\item  the geodesic trajectory is wandering to infinity, i.e.~permanently escapes every compact subset.
\end{enumerate}
Note that properties (1)-(2) are $ U $-invariant. Hence ergodicity implies that either $ \mu $-a.e.~geodesic trajectory is recurrent or $ \mu $-a.e.~trajectory is wandering.

In case of recurrence, we have for $ \mu $-a.e.~$ g\Gamma \in \GmodGamma $ a subsequence of times $ t_n \to \infty $ and an element $ h \in G $ for which $ a_{-t_n}g\Gamma \to h\Gamma $. Consequently
\[ a_{-t_n}g\Gamma g^{-1}a_{t_n} \longrightarrow h\Gamma h^{-1} \]
in the sense of Hausdorff convergence inside compact subsets of $ G $. Clearly $ h \Gamma h^{-1} $ is a Zariski dense subgroup of $ G $ (since $ \Gamma $ was assumed to be Zariski dense) and also $ h \Gamma h^{-1} \subseteq \Xi_{g\Gamma} $, implying $ \mu $ is $ MA $-quasi-invariant.
\medskip

Now assume $ \left( p(a_{-t}g\Gamma) \right)_{t \geq 0} $ is wandering for $ \mu $-a.e.~$ g\Gamma $. Fix 
\[ 0 < \epsilon < \min\{\rho_3,I_\mu\} \]
where $\rho_3$ is as in \S\ref{subsec_noncuspidal}. and let $ \mathcal{P} $ and $ C_{rel} $ be the corresponding cuspidal part and relative compact core of $ \M $. Let $ \{ E_1,...,E_\ell \} $ be the induced collection of relative ends of $ \M $. Recall that
\[ \M = (C_{rel} \cup \mathcal{P}) \sqcup \bigsqcup_{i=1}^{\ell} E_i. \]
For $ \mu $-a.e.~$ g\Gamma $ there exists a sequence $ t_n \to \infty $ with $ d_{\M}(p(a_{-t_n}x),C_{rel}) \to \infty $ and $ \injM(p(a_{-t_n}g\Gamma)) > \epsilon $ therefore by taking a subsequence we may assume there exists an $ 1 \leq i \leq \ell $ such that the points $ p(a_{-t_n}g\Gamma) $ are all contained inside $ E_i $. Furthermore, for any sub-neighborhood $ V $ in $ E_i $ in the sense of \S\ref{subsec_RelativeEnds}, the sequence $ (p(a_{-t_n}g\Gamma))_n $ is contained in $ V $ for all large $ n $. Since $ I_\mu < \infty $ we know the geodesic ray stays within a bounded distance of the convex core of $ \M $ (otherwise its lifts in $ \Hspace $ would tend to a limit point outside the limit set ensuring the injectivity radius tends to infinity). Hence $ E_i $ is a geometrically infinite end, implying $ p(a_{-t_n}g\Gamma) \in C(\M) $ for all large $ n $.

By taking a subsequence of times, we may assume that $ p(a_{-t_n}g\Gamma) \in C(\M) $ for all $ n $ and that the sequence of sets $ a_{-t_n}g\Gamma g^{-1}a_{t_n} $ converges to a closed subset $ \Sigma $ of $ G $ (as before in the sense of Hausdorff convergence inside compact subsets of $ G $). $ \Sigma $ is clearly a subgroup of $ G $ and the lower bound on the injectivity radius at $ p(a_{-t_n}g\Gamma) $ ensures $ \Sigma $ is discrete (recall \eqref{eq_injGamma}). 

Let $ \kappa > 0 $ be the upper bound on the injectivity radius in $ C(\M) $ ensured by Canary's theorem. Assume in contradiction that $ \Sigma $ is not Zariski dense in $ G $. It is hence a discrete subgroup of the second kind, implying by \cref{lemma_unbounded inj rad} that $ \Hspace/\Sigma $ has unbounded injectivity radius. Therefore there exists a constant $ R > 0 $ for which $ \Hspace / \Sigma $ contains an embedded hyperbolic ball of radius $ \kappa+1 $ contained inside a neighborhood of radius $ R $ around the identity coset in  $\Hspace / \Sigma $. This in turn implies that for all large $ n $ the radius $ R $ neighborhood of the point $ p(a_{-t_n}g\Gamma) $ in $ \M $ contains an embedded ball of radius greater than $ \kappa+\frac{1}{2} $ (see e.g.~\cite[Thm. 7.7]{Matsuzaki-Taniguchi}). Denote by $ y_n $ the center of said embedded ball. 

Clearly 
\[ y_n \in \M^0 \smallsetminus C(\M) = \M \smallsetminus (\mathcal{P} \cup C(\M)), \]
since
\[ \injM(y_n) > \kappa + \frac{1}{2} > \sup_{x \in C(\M)} \injM(x) \geq \varepsilon. \]
The geodesic arc
\[ \psi = [y_n,p(a_{-t_n}g\Gamma)] \]
connecting $ y_n $ with 
\[ p(a_{-t_n}g\Gamma) \in C(\M) \]
intersects $ \partial C(\M) $. This arc is also contained inside $ (\M^0)^{(R)} $, the $ R $-neighborhood of $ \M^0 $ (the non-cuspidal part), since $ d_\M \left( p(a_{-t_n}g\Gamma) , y_n  \right) < R $. 

Therefore $ \psi $ intersects the set
\[ Y = \partial C(\M) \cap \overline{(\M^0)^{(R)}}, \]
and consequently
\begin{equation*}
d_\M \left( p(a_{-t_n}g\Gamma) , Y  \right) < R
\end{equation*}
for all large $ n $.
But the Ahlfors finiteness theorem ensures $ Y $ is compact (see proof of \cref{lemma_infinite end inside convex core}), contradicting the assumption that the sequence $ \left( p(a_{-t}g\Gamma) \right)_{t \geq 0} $ is wandering.

Therefore $ \Sigma \subseteq \Xi_{g\Gamma} $ is a Zariski dense subgroup of $ G $ and by \cref{Thm_MA-qi} the measure $ \mu $ is $ MA $-quasi-invariant, concluding the proof.
\end{proof}

\begin{rem}
	The arguments in the proof of \cref{Main_thm}, specifically the argument used in the case where geodesic trajectories wander to infinity through a geometrically infinite end, may be used almost verbatim in any dimension to prove \cref{Thm_d-dim}.
\end{rem}
\medskip
\medskip

We conclude this section with the proof of \cref{Cor_P measures}:

\begin{proof}[Proof of \cref{Cor_P measures}]
	Let $ \nu $ be a $ U $-invariant measure on $ \GmodGamma $ which is $ P $-quasi-invariant and ergodic. Let
	\[
	\nu=\int_Y \!\nu_\xi\, d\rho(\xi)
	\]
	be the ergodic decomposition of the measure $ \nu $ with respect to the $ U $-action, where $Y$ is a standard Borel space, $\rho$ a probability measure on $Y$, and $\nu_\xi$ a $U$-invariant and ergodic locally finite measure on $\GmodGamma$. Such an ergodic decomposition can be found e.g.~in \cite{Greschonig-Schmidt} (there are several versions of the ergodic decomposition in that paper, but one can for instance apply \cite[Theorem 1.1]{Greschonig-Schmidt} and then conclude that each ergodic component is actually $U$-invariant); alternatively one can proceed as in \cite[Thm. 2.2.8]{Aaronson_Book}, substituting Hochman's ratio ergodic theorem \cite{Hochman2010} instead of the classical Hurewicz ratio ergodic theorem.
	
	By \cref{Main_thm}, for every $\xi$ the ergodic component $\nu_\xi$ is either trivial or $MA$-quasi invariant. 
	\medskip
	
	Let $ \pi : G \to P\backslash G $ be the projection onto $ P\backslash G \cong M \backslash K \cong \partial \Hspace $. Identifying $ G $ with $ \FHspace $, the frame bundle of hyperbolic $ 3 $-space, this quotient map corresponds to the projection of frames onto the limit points in $ \partial \Hspace $ they tend to under the geodesic flow $ a_{-t} $ as $ t \to +\infty $. 
	
	Let $ \Lambda $ denote the limit set of $ \Gamma $ in $ P \backslash G $ and let $ \Lambda_{\text{par}} $ denote the subset of parabolic fixed points. The $P$-ergodicity of $ \nu $ ensures that either 
	\begin{enumerate}
		\item $\nu_\xi$ is $ MA $-quasi-invariant for $\rho$-a.e. $\xi$;
		\item $\nu_\xi$ is trivial and supported on $ \pi^{-1}(\Lambda_{\text{par}})/\Gamma $ for $\rho$-a.e. $\xi$; or
		\item $\nu_\xi$ is trivial and supported on $ \pi^{-1}(P \backslash G \smallsetminus \Lambda)/ \Gamma $ for $\rho$-a.e. $\xi$.
	\end{enumerate}
	\noindent This follows from the fact that both $ \Lambda_{\text{par}} $ and $ P\backslash G \smallsetminus \Lambda $ are $ P $-invariant, ensuring the respective sets in $ \GmodGamma $ are either null or conull with respect to $ \nu $. Whenever both sets are $ \nu $-null then \cref{Main_thm} implies case (1) is the only possible alternative.
	
	Note that if (1) holds then any $U$-invariant set is (up to a null set) also $P$ invariant. Hence by $P$ ergodicity the $\sigma$-algebra of $U$-invariant sets is equivalent mod $\nu$ to the trivial $\sigma$-algebra, i.e. $\nu$ is $ U $-ergodic (hence $ \nu_\xi = \nu$ for $\rho $-a.e.~$\xi$).
	\medskip
	
	We claim the remaining two alternatives imply $ \nu $ is supported on a single $ P $-orbit. Indeed recall that $ \Lambda_{\text{par}} \subseteq P \backslash G $ is countable, hence if $ \nu $ gives positive measure to $ \pi^{-1}(\Lambda_{\text{par}})/\Gamma $ there must exist a single $ P $-orbit receiving positive mass. $ P $-ergodicity of $ \nu $ implies this is the entire support of $ \nu $.
	
	Assume the third alternative , i.e.~that $ \nu $-a.e.~$ P $-orbit is based outside the limit set. By the $ P $-ergodicity of $\nu$, there exists a point $ g\Gamma \in \GmodGamma $ whose $ P $-orbit is dense in $ \mathrm{supp}(\nu) $ (the topological support of $ \nu $). Assume in contradiction that there exists a point $ h\Gamma \in \mathrm{supp}(\nu) $ with a $ P $-orbit based outside the limit set of $ \Gamma $ and satisfying $ Pg\Gamma \neq Ph\Gamma $. 
	
	Since $ \Gamma $ acts properly discontinuously on $ P \backslash G \smallsetminus \Lambda $ and since $ Pg\Gamma \neq Ph\Gamma $, there exists an open neighborhood $ W \subset P \backslash G $ of $ \pi(h) $ for which
	\begin{equation*}
	\pi(g)\Gamma \cap W\Gamma = \emptyset.
	\end{equation*}
	But this in turn implies that
	\[ Pg\Gamma \cap \pi^{-1}(W)\Gamma = \emptyset \]
	where $ \pi^{-1}(W)\Gamma $ is an open neighborhood of the point $ h\Gamma $ in $ \GmodGamma $, contradicting the fact that $ Pg\Gamma $ is dense in $ \mathrm{supp}(\nu) $. Therefore, $ \nu $ is supported on a single $ P $-orbit as claimed.
\end{proof}

\end{document}